\documentclass[11pt]{article}

\usepackage[margin=1in]{geometry}
\usepackage[T1]{fontenc}
\usepackage[utf8]{inputenc}
\usepackage{lmodern}
\usepackage{amsmath,amssymb,amsthm,mathtools}
\usepackage[hidelinks]{hyperref}

\numberwithin{equation}{section}

\newtheorem{theorem}{Theorem}[section]
\newtheorem{lemma}[theorem]{Lemma}
\newtheorem{corollary}[theorem]{Corollary}
\newtheorem{proposition}[theorem]{Proposition}
\theoremstyle{definition}
\newtheorem{definition}[theorem]{Definition}
\theoremstyle{remark}
\newtheorem{remark}[theorem]{Remark}

\title{Distributional Point Values for Borel and Symmetric Borel Derivatives}
\author{Subhasis Ray\\
Department of Mathematics, Siksha Bhavana, Visva-Bharati,\\
Santiniketan 731235, West Bengal, India\\
\texttt{subhasis.ray@visva-bharati.ac.in}}
\date{}

\begin{document}
\maketitle

\begin{abstract}
Borel and symmetric Borel derivatives are generalized derivatives defined through local averages
of difference quotients.
Distributional point values, in the sense of {\L}ojasiewicz and its symmetric variants, are a
classical way of describing the local value of a distribution.
This paper connects these two ideas.
Writing $T_f$ for the regular distribution generated by $f$, we prove that finite first and
second symmetric Borel derivatives give symmetric distributional point values of $T_f'$ and
$T_f''$, respectively.
For the first symmetric derivative, Borel smoothness is used as a sufficient condition to pass from
the symmetric point value to the full {\L}ojasiewicz point value.
We also prove that the one-sided Borel derivatives determine the right and left distributional point
values of $T_f'$, and that the ordinary Borel derivative gives the full point value when the two
one-sided averages agree.
Examples show why the second-order symmetric result cannot be strengthened automatically.
\end{abstract}



\section{Introduction}
The Borel derivative was introduced by E.~Borel, who called it the average derivative.
Its first-order theory was studied by A.~Khintchine \cite{Khintchine1927},
J.~Marcinkiewicz and A.~Zygmund \cite{MarcinkiewiczZygmund1936},
and W.L.C.~Sargent \cite{Sargent1935}.
The Borel derivative is a generalized derivative: it may exist even when the ordinary derivative
fails to exist.
Mukhopadhyay's monograph \cite{Mukhopadhyay2012} later treated Borel and symmetric
Borel derivatives as part of the wider theory of higher-order generalized derivatives.
Ray and Garai \cite{RayGarai2014} subsequently studied the first-order Borel derivative from an
elementary real-variable point of view.

The distributional point value is another classical local notion.
It begins with the work of {\L}ojasiewicz \cite{Lojasiewicz1957} and has been developed further in
terms of symmetric point values, lateral limits, jump behavior, and delta sequences
\cite{VindasEstrada2007Fourier,VindasEstrada2007Reg,VindasEstrada2008Jump,VindasEstrada2010,EstradaKellinsky2022}.
These point values describe when a distribution has a meaningful value at a point, even though
ordinary pointwise evaluation is not available for general distributions.

The purpose of this paper is to connect these two theories.
The basic theory of Borel and symmetric Borel derivatives is already available in the generalized
derivative literature, and the general theory of distributional point values is also well developed.
What is new here is the bridge between them: we show that Borel-type averaged difference
quotients give symmetric, one-sided, or full distributional point values of the corresponding
distributional derivatives.

The reason this bridge works is simple.
The kernels that occur in Borel averages have the same averaging character as the smooth kernels
used to test distributions.
After rewriting the distributional pairings with $T_f'$ and $T_f''$, the Borel averages appear as
weighted mollifier averages of the same local difference quotients.
This allows us to pass from the existence of Borel-type derivatives to distributional point values.

For symmetric Borel derivatives, the first and second order results give symmetric point values of
$T_f'$ and $T_f''$.
In the first-order case, Borel smoothness is a sufficient local condition that removes the odd part
of the test functions and gives the full {\L}ojasiewicz point value of $T_f'$.
For the ordinary Borel derivative, the one-sided Borel averages give right and left distributional
point values, and their agreement gives the full point value.
Examples show why these statements cannot be automatically strengthened, especially in the
second-order case.

\section{Preliminaries}

Throughout the paper, $x_0$ denotes a fixed point of $\mathbb{R}$ and $f$ is a real-valued function
defined in a neighborhood of $x_0$.
Whenever $f\in L^1_{\mathrm{loc}}(\mathbb{R})$, we work with a fixed pointwise representative of $f$
near $x_0$. The quantities $D_1$, $D_2$, $B_\pm$, $SBD^1$, $SBD^2$, and $BD_1^\pm$ depend on that
representative, whereas the regular distribution $T_f$ does not.
For $t\neq 0$ we set
\[
D_1(t;x_0;f)=\frac{f(x_0+t)-f(x_0-t)}{2t},
\qquad
D_2(t;x_0;f)=\frac{f(x_0+t)+f(x_0-t)-2f(x_0)}{t^2}.
\]

\begin{definition}
The first and second symmetric Borel derivatives of $f$ at $x_0$ are defined, whenever the limits exist finitely, by
\begin{align*}
SBD^{1}f(x_0)&=\lim_{h\to 0^+}\frac{1}{h}\int_0^h D_1(t;x_0;f)\,dt,\\
SBD^{2}f(x_0)&=\lim_{h\to 0^+}\frac{1}{h}\int_0^h D_2(t;x_0;f)\,dt.
\end{align*}
\end{definition}

\begin{definition}
The function $f$ is said to be \emph{Borel smooth} at $x_0$ if
\[
\lim_{h\to 0^+}\frac{1}{h}\int_0^h
\frac{f(x_0+t)+f(x_0-t)-2f(x_0)}{t}\,dt =0.
\]
\end{definition}

The standard theory of symmetric Borel derivatives and Borel smoothness,
together with their relations to other higher-order generalized derivatives,
can be found in \cite{Mukhopadhyay2012}. The definitions above are the only
facts from that theory needed below.

\section{Distributional Point Values for Borel Averages}\label{sec:dpv}

This section contains the main results of the paper.
We first study the first symmetric Borel derivative and its full point-value upgrade under
Borel smoothness.
We then show that the ordinary Borel derivative naturally gives right and left distributional
point values.
After that we return to the second symmetric Borel derivative and to the parity obstruction
coming from point-supported singularities.

\begin{definition}
Let $\mathcal{D}(\mathbb{R})=C_c^\infty(\mathbb{R})$ and let $\mathcal{D}'(\mathbb{R})$ be the
space of distributions.
If $f\in L^1_{\mathrm{loc}}(\mathbb{R})$, we identify $f$ with the regular distribution
$T_f\in\mathcal{D}'(\mathbb{R})$ defined by
\[
\langle T_f,\psi\rangle=\int_{\mathbb{R}} f(x)\psi(x)\,dx,
\qquad \psi\in\mathcal{D}(\mathbb{R}).
\]
Its first and second distributional derivatives are denoted by $T_f'$ and $T_f''$, that is,
\[
\langle T_f',\psi\rangle=-\langle T_f,\psi'\rangle,
\qquad
\langle T_f'',\psi\rangle=\langle T_f,\psi''\rangle,
\qquad \psi\in\mathcal{D}(\mathbb{R}).
\]
\end{definition}

\begin{definition}
Let $T\in\mathcal{D}'(\mathbb{R})$ and $x_0\in\mathbb{R}$.
We say that $T$ has the \emph{{\L}ojasiewicz distributional point value} $L\in\mathbb{R}$ at $x_0$ if
\[
\lim_{\varepsilon\to 0^+}\langle T,\phi_{x_0,\varepsilon}\rangle
=L\int_{\mathbb{R}}\phi(u)\,du
\qquad \text{for every } \phi\in\mathcal{D}(\mathbb{R}),
\]
where
\[
\phi_{x_0,\varepsilon}(x)=\varepsilon^{-1}\phi\!\left(\frac{x-x_0}{\varepsilon}\right).
\]
\end{definition}

\begin{definition}
Let $T\in\mathcal{D}'(\mathbb{R})$ and $x_0\in\mathbb{R}$.
We say that $T$ has the \emph{symmetric distributional point value} $L\in\mathbb{R}$ at $x_0$ if
for every even $\phi\in\mathcal{D}(\mathbb{R})$,
\[
\lim_{\varepsilon\to 0^+}\langle T,\phi_{x_0,\varepsilon}\rangle
=L\int_{\mathbb{R}}\phi(u)\,du.
\]
In that case we write $\mathrm{sv}\,T(x_0)=L$.
\end{definition}
\begin{remark}\label{rem:symmetric-pv}
The preceding definition is the symmetric version of the usual distributional point value.
In one dimension, it is equivalent to testing against all even test functions,
or equivalently against the family of standard delta sequences generated by
positive normalized even test functions \cite{EstradaKellinsky2022}.
If the full {\L}ojasiewicz point value of $T$ exists at $x_0$, then the symmetric
point value also exists and agrees with it. The converse need not hold.
\end{remark}

\begin{lemma}\label{lem:kernel-representation} Change korte hobe
Let $f\in L^1_{\mathrm{loc}}(\mathbb{R})$ and let $\phi\in\mathcal{D}(\mathbb{R})$ be even.
Define, for $u\ge 0$,
\[
W_1(u)=-2u\,\phi'(u),
\qquad
W_2(u)=u^2\,\phi''(u).
\]
Then $W_1,W_2\in C_c^1([0,\infty))$ and
\[
\int_0^\infty W_1(u)\,du=\int_{\mathbb{R}}\phi(u)\,du,
\qquad
\int_0^\infty W_2(u)\,du=\int_{\mathbb{R}}\phi(u)\,du.
\]
Moreover, for every $\varepsilon>0$,
\begin{align}
\langle T_f',\phi_{x_0,\varepsilon}\rangle
&=\int_0^\infty D_1(\varepsilon u;x_0;f)\,W_1(u)\,du,\label{eq:pairing-fprime}\\[1mm]
\langle T_f'',\phi_{x_0,\varepsilon}\rangle
&=\int_0^\infty D_2(\varepsilon u;x_0;f)\,W_2(u)\,du.\label{eq:pairing-fdoubleprime}
\end{align}
\end{lemma}

\begin{proof}
Clearly $W_1,W_2\in C_c^1([0,\infty))$, $\phi'$ is odd, and $\phi'(0)=0$.
Using integration by parts 
\begin{align*}
\int_0^\infty W_1(u)\,du
&=\int_0^\infty (-2u\phi'(u))\,du
=2\int_0^\infty \phi(u)\,du
=\int_{\mathbb{R}}\phi(u)\,du,
\end{align*}
and similarly,
\begin{align*}
\int_0^\infty W_2(u)\,du
&=\int_0^\infty u^2\phi''(u)\,du
=\Big[u^2\phi'(u)\Big]_0^\infty-\int_0^\infty 2u\phi'(u)\,du\\
&=-2\int_0^\infty u\phi'(u)\,du
=-2\Big(\Big[u\phi(u)\Big]_0^\infty-\int_0^\infty \phi(u)\,du\Big)\\
&=2\int_0^\infty \phi(u)\,du
=\int_{\mathbb{R}}\phi(u)\,du.
\end{align*}

By definition of distributional derivative,
\[
\langle T_f',\phi_{x_0,\varepsilon}\rangle
=-\int_{\mathbb{R}}f(x)(\phi_{x_0,\varepsilon})'(x)\,dx.
\]
Since $(\phi_{x_0,\varepsilon})'(x)=\varepsilon^{-2}\phi'((x-x_0)/\varepsilon)$, the change of
variable $x=x_0+\varepsilon u$ gives
\[
\langle T_f',\phi_{x_0,\varepsilon}\rangle
=-\frac{1}{\varepsilon}\int_{\mathbb{R}}f(x_0+\varepsilon u)\phi'(u)\,du.
\]
Because $\phi'$ is odd,
\[
\int_{\mathbb{R}}f(x_0+\varepsilon u)\phi'(u)\,du
=\int_0^\infty \big(f(x_0+\varepsilon u)-f(x_0-\varepsilon u)\big)\phi'(u)\,du,
\]
and therefore
\[
\langle T_f',\phi_{x_0,\varepsilon}\rangle
=-\frac{1}{\varepsilon}\int_0^\infty
\big(f(x_0+\varepsilon u)-f(x_0-\varepsilon u)\big)\phi'(u)\,du
=\int_0^\infty D_1(\varepsilon u;x_0;f)\,(-2u\phi'(u))\,du,
\]
which proves \eqref{eq:pairing-fprime}.

For \eqref{eq:pairing-fdoubleprime}, we use $\langle T_f'',\psi\rangle=\langle T_f,\psi''\rangle$:
\[
\langle T_f'',\phi_{x_0,\varepsilon}\rangle
=\int_{\mathbb{R}}f(x)(\phi_{x_0,\varepsilon})''(x)\,dx
=\frac{1}{\varepsilon^2}\int_{\mathbb{R}}f(x_0+\varepsilon u)\phi''(u)\,du.
\]
Since $\phi''$ is even,
\[
\int_{\mathbb{R}}f(x_0+\varepsilon u)\phi''(u)\,du
=\int_0^\infty \big(f(x_0+\varepsilon u)+f(x_0-\varepsilon u)\big)\phi''(u)\,du.
\]
Also $\int_0^\infty \phi''(u)\,du=\phi'(\infty)-\phi'(0)=0$, so we may insert
$-2f(x_0)$ inside the bracket without changing the integral.
Thus
\begin{align*}
\langle T_f'',\phi_{x_0,\varepsilon}\rangle
&=\frac{1}{\varepsilon^2}\int_0^\infty
\big(f(x_0+\varepsilon u)+f(x_0-\varepsilon u)-2f(x_0)\big)\phi''(u)\,du\\
&=\int_0^\infty D_2(\varepsilon u;x_0;f)\,(u^2\phi''(u))\,du,
\end{align*}
which proves \eqref{eq:pairing-fdoubleprime}.
\end{proof}

\begin{lemma}\label{lem:abelian} Change korte hobe
Let $g\in L^1_{\mathrm{loc}}(0,\infty)$ and define
\[
G(h)=\frac{1}{h}\int_0^h g(t)\,dt,
\qquad h>0.
\]
Assume that $\lim_{h\to 0^+}G(h)=L$ is finite.
Then for every $W\in C_c^1([0,\infty))$,
\[
\lim_{\varepsilon\to 0^+}\int_0^\infty g(\varepsilon u)W(u)\,du
=L\int_0^\infty W(u)\,du.
\]
\end{lemma}

\begin{proof}
Let $W\in C_c^1([0,\infty))$ and choose $A>0$ such that ${\rm supp}(W)\subset[0,A]$.
Set $F(t)=\int_0^t g(s)\,ds$.
Then $F$ is absolutely continuous and $F(t)=tG(t)$.
For fixed $\varepsilon>0$, define
\[
H_\varepsilon(u)=\int_0^u g(\varepsilon s)\,ds
=\frac{1}{\varepsilon}\int_0^{\varepsilon u}g(t)\,dt
=uG(\varepsilon u).
\]
By integration by parts,
\begin{align*}
\int_0^\infty g(\varepsilon u)W(u)\,du
&=\int_0^A g(\varepsilon u)W(u)\,du\\
&=\Big[H_\varepsilon(u)W(u)\Big]_0^A-\int_0^A H_\varepsilon(u)W'(u)\,du\\
&=-\int_0^A uG(\varepsilon u)W'(u)\,du,
\end{align*}
since $W(A)=0$ and $H_\varepsilon(0)=0$.
As $\varepsilon\to 0^+$, we have $G(\varepsilon u)\to L$ for each $u\in[0,A]$, and
$G$ is bounded near $0$ because it converges.
Therefore dominated convergence gives
\[
\lim_{\varepsilon\to 0^+}\int_0^\infty g(\varepsilon u)W(u)\,du
=-L\int_0^A uW'(u)\,du.
\]
A second integration by parts yields
\[
-\int_0^A uW'(u)\,du
=-\Big[uW(u)\Big]_0^A+\int_0^A W(u)\,du
=\int_0^\infty W(u)\,du.
\]
This completes the proof.
\end{proof}

\subsection{The first symmetric Borel derivative}

\begin{theorem}\label{th:SBD1-dpv}
Let $f\in L^1_{\mathrm{loc}}(\mathbb{R})$ and let $x_0\in\mathbb{R}$.
If $SBD^{1}f(x_0)$ exists finitely, say $SBD^{1}f(x_0)=L$, then for every even
$\phi\in\mathcal{D}(\mathbb{R})$,
\[
\lim_{\varepsilon\to 0^+}\langle T_f',\phi_{x_0,\varepsilon}\rangle
=L\int_{\mathbb{R}}\phi(u)\,du.
\]
In particular, $T_f'$ has symmetric distributional point value $L$ at $x_0$.
\end{theorem}

\begin{proof}
Fix an even $\phi\in\mathcal{D}(\mathbb{R})$ and let $W_1$ be as in
Lemma \ref{lem:kernel-representation}.
Put $g(t)=D_1(t;x_0;f)$ for $t>0$.
By definition of $SBD^{1}f(x_0)$,
\[
\lim_{h\to 0^+}\frac{1}{h}\int_0^h g(t)\,dt=L.
\]
Since $\int_0^\infty W_1(u)\,du=\int_{\mathbb{R}}\phi(u)\,du$,
Lemma \ref{lem:abelian} gives
\[
\lim_{\varepsilon\to 0^+}\int_0^\infty g(\varepsilon u)W_1(u)\,du
=L\int_{\mathbb{R}}\phi(u)\,du.
\]
By Lemma \ref{lem:kernel-representation}, the integral on the left equals
$\langle T_f',\phi_{x_0,\varepsilon}\rangle$.
Hence the theorem follows.
\end{proof}

Theorem \ref{th:SBD1-dpv} captures the contribution of the even part of the test family.
A natural additional hypothesis is Borel smoothness. In the first-order case it controls the odd
part of the test family and therefore upgrades the symmetric point value to the full
{\L}ojasiewicz point value.

\begin{theorem}\label{th:SBD1-loj}
Let $f\in L^1_{\mathrm{loc}}(\mathbb{R})$ and let $x_0\in\mathbb{R}$.
Assume that $SBD^{1}f(x_0)=L$ and that $f$ is Borel smooth at $x_0$.
Then $T_f'$ has the {\L}ojasiewicz distributional point value $L$ at $x_0$.
Equivalently,
\[
\lim_{\varepsilon\to 0^+}\langle T_f',\phi_{x_0,\varepsilon}\rangle
=L\int_{\mathbb{R}}\phi(u)\,du
\qquad \text{for every } \phi\in\mathcal{D}(\mathbb{R}).
\]
\end{theorem}

\begin{proof}
Let $\phi\in\mathcal{D}(\mathbb{R})$ be arbitrary and write
\[
\phi=\phi_e+\phi_o,
\qquad
\phi_e(u)=\frac{\phi(u)+\phi(-u)}{2},
\qquad
\phi_o(u)=\frac{\phi(u)-\phi(-u)}{2}.
\]
Since $\phi_e$ is even by Theorem \ref{th:SBD1-dpv} 
\[
\lim_{\varepsilon\to 0^+}\langle T_f',(\phi_e)_{x_0,\varepsilon}\rangle
=L\int_{\mathbb{R}}\phi_e(u)\,du.
\]

For the odd part we have $\phi_o'$ is even.
Therefore
\begin{align*}
\langle T_f',(\phi_o)_{x_0,\varepsilon}\rangle
&=-\frac{1}{\varepsilon}\int_{\mathbb{R}}f(x_0+\varepsilon u)\phi_o'(u)\,du\\
&=-\frac{1}{\varepsilon}\int_0^\infty
\big(f(x_0+\varepsilon u)+f(x_0-\varepsilon u)\big)\phi_o'(u)\,du.
\end{align*}
Since $\phi_o(0)=0$ and $\phi_o$ has compact support,
\[
\int_0^\infty \phi_o'(u)\,du=\phi_o(\infty)-\phi_o(0)=0.
\]
Thus we get
\[
\langle T_f',(\phi_o)_{x_0,\varepsilon}\rangle
=-\frac{1}{\varepsilon}\int_0^\infty
\big(f(x_0+\varepsilon u)+f(x_0-\varepsilon u)-2f(x_0)\big)\phi_o'(u)\,du.
\]
Let
\[
g(t)=\frac{f(x_0+t)+f(x_0-t)-2f(x_0)}{t},
\qquad
W(u)=-u\phi_o'(u).
\]
Then $W\in C_c^1([0,\infty))$ and
\[
\langle T_f',(\phi_o)_{x_0,\varepsilon}\rangle
=\int_0^\infty g(\varepsilon u)W(u)\,du.
\]
Since $f$ is Borel smooth at $x_0$,
\[
\lim_{h\to 0^+}\frac{1}{h}\int_0^h g(t)\,dt=0.
\]
Applying Lemma \ref{lem:abelian} with $L=0$, we get
\[
\lim_{\varepsilon\to 0^+}\langle T_f',(\phi_o)_{x_0,\varepsilon}\rangle=0.
\]
Hence
\begin{align*}
\lim_{\varepsilon\to 0^+}\langle T_f',\phi_{x_0,\varepsilon}\rangle
&=\lim_{\varepsilon\to 0^+}\langle T_f',(\phi_e)_{x_0,\varepsilon}\rangle
+\lim_{\varepsilon\to 0^+}\langle T_f',(\phi_o)_{x_0,\varepsilon}\rangle\\
&=L\int_{\mathbb{R}}\phi_e(u)\,du
=L\int_{\mathbb{R}}\phi(u)\,du,
\end{align*}
because $\int_{\mathbb{R}}\phi_o(u)\,du=0$.
This completes the proof.
\end{proof}

\begin{proposition}
Borel smoothness in Theorem \ref{th:SBD1-loj} is essential:
For $f(x)=|x|$ one has
\[
SBD^{1}f(0)=0,
\]
but $f$ is not Borel smooth at $0$ and $T_f'=\operatorname{sgn}$ has no {\L}ojasiewicz
point value at $0$.
\end{proposition}

\begin{proof}
For $t>0$,
\[
D_1(t;0;f)=\frac{|t|-|-t|}{2t}=0,
\]
so $SBD^{1}f(0)=0$.
On the other hand,
\[
\frac{1}{h}\int_0^h\frac{f(t)+f(-t)-2f(0)}{t}\,dt
=\frac{1}{h}\int_0^h 2\,dt=2,
\]
so $f$ is not Borel smooth at $0$.
Since the distributional derivative of $|x|$ is the regular distribution $\operatorname{sgn}$,
we have for every $\phi\in\mathcal{D}(\mathbb{R})$,
\[
\langle T_f',\phi_{0,\varepsilon}\rangle
=\int_{\mathbb{R}}\operatorname{sgn}(u)\phi(u)\,du.
\]
This vanishes for every even $\phi$, but it does not equal
$0\cdot \int_{\mathbb{R}}\phi(u)\,du$ for arbitrary $\phi$.
Thus $T_f'$ has symmetric distributional point value $0$ at $0$, but it has no full
{\L}ojasiewicz point value there.
\end{proof}

\subsection{Point values associated with the Borel derivative}\label{subsec:bd-point-values}

The one-sided Borel averages can be handled in the same way.
They lead naturally to right and left distributional point values of $T_f'$.

For $t>0$ define
\[
B_+(t;x_0;f)=\frac{f(x_0+t)-f(x_0)}{t},
\qquad
B_-(t;x_0;f)=\frac{f(x_0)-f(x_0-t)}{t}.
\]
Then
\[
BD_1^+f(x_0)=\lim_{h\to0^+}\frac{1}{h}\int_0^h B_+(t;x_0;f)\,dt,
\qquad
BD_1^-f(x_0)=\lim_{h\to0^+}\frac{1}{h}\int_0^h B_-(t;x_0;f)\,dt,
\]
whenever these limits exist finitely.
If $BD_1^+f(x_0)=BD_1^-f(x_0)$, then their common value is denoted by $BD_1f(x_0)$ and is
called the ordinary Borel derivative of $f$ at $x_0$.

\begin{definition}
Let
\[
\mathcal{D}_+(\mathbb{R})=\{\phi\in\mathcal{D}(\mathbb{R}) : \operatorname{supp}\phi\subset [0,\infty)\},
\qquad
\mathcal{D}_-(\mathbb{R})=\{\phi\in\mathcal{D}(\mathbb{R}) : \operatorname{supp}\phi\subset (-\infty,0]\}.
\]
We say that $T\in\mathcal{D}'(\mathbb{R})$ has \emph{right distributional point value} $L_+$ at $x_0$ if
\[
\lim_{\varepsilon\to0^+}\langle T,\phi_{x_0,\varepsilon}\rangle
=
L_+\int_{\mathbb{R}}\phi(u)\,du
\qquad \text{for every }\phi\in\mathcal{D}_+(\mathbb{R}),
\]
and \emph{left distributional point value} $L_-$ at $x_0$ if
\[
\lim_{\varepsilon\to0^+}\langle T,\phi_{x_0,\varepsilon}\rangle
=
L_-\int_{\mathbb{R}}\phi(u)\,du
\qquad \text{for every }\phi\in\mathcal{D}_-(\mathbb{R}).
\]
\end{definition}

These one-sided point values are local versions of distributional lateral limits.
In the notation of the next theorem, the limiting distribution is
\(L_-H(-u)+L_+H(u)\), so the result may also be viewed as a jump-behavior
statement of the type studied in \cite{VindasEstrada2007Reg,VindasEstrada2008Jump}.

\begin{theorem}\label{th:bd-point-value-splitting} Let $f\in L^1_{\mathrm{loc}}(\mathbb{R})$ and let $x_0\in\mathbb{R}$.
Assume that $BD_1^+f(x_0)=L_+$ and $BD_1^-f(x_0)=L_-$ both exist finitely.
Then, for every $\phi\in\mathcal{D}(\mathbb{R})$,
\[
\lim_{\varepsilon\to0^+}\langle T_f',\phi_{x_0,\varepsilon}\rangle
=
L_+\int_0^\infty \phi(u)\,du
+
L_-\int_{-\infty}^0 \phi(u)\,du.
\]
Consequently, $T_f'$ has right distributional point value $L_+$ and left distributional point value
$L_-$ at $x_0$.
Moreover, $T_f'$ has the {\L}ojasiewicz distributional point value at $x_0$ if and only if
$L_+=L_-$; in that case the point value equals this common value.
\end{theorem}

\begin{proof}
Let $\phi\in\mathcal{D}(\mathbb{R})$ be arbitrary.
Since $\int_{\mathbb{R}}\phi'(u)\,du=0$, we may write
\begin{align*}
\langle T_f',\phi_{x_0,\varepsilon}\rangle
&=
-\frac{1}{\varepsilon}\int_{\mathbb{R}}[f(x_0+\varepsilon u)-f(x_0)]\phi'(u)\,du\\
&=I_+(\varepsilon)+I_-(\varepsilon),
\end{align*}
where
\[
I_+(\varepsilon)
=
-\frac{1}{\varepsilon}\int_0^\infty [f(x_0+\varepsilon u)-f(x_0)]\phi'(u)\,du
=
\int_0^\infty B_+(\varepsilon u;x_0;f)\,W_+(u)\,du
\]
with
\[
W_+(u)=-u\phi'(u),
\]
and
\begin{align*}
I_-(\varepsilon)
&=
-\frac{1}{\varepsilon}\int_{-\infty}^0 [f(x_0+\varepsilon u)-f(x_0)]\phi'(u)\,du\\
&=
-\frac{1}{\varepsilon}\int_0^\infty [f(x_0-\varepsilon u)-f(x_0)]\phi'(-u)\,du\\
&=
\int_0^\infty B_-(\varepsilon u;x_0;f)\,W_-(u)\,du
\end{align*}
with
\[
W_-(u)=u\phi'(-u).
\]

Now $W_+,W_-\in C_c^1([0,\infty))$, and integration by parts gives
\[
\int_0^\infty W_+(u)\,du=\int_0^\infty \phi(u)\,du,
\qquad
\int_0^\infty W_-(u)\,du=\int_{-\infty}^0 \phi(u)\,du.
\]
Applying Lemma \ref{lem:abelian} to $B_+$ and $B_-$, we obtain
\[
I_+(\varepsilon)\to L_+\int_0^\infty \phi(u)\,du,
\qquad
I_-(\varepsilon)\to L_-\int_{-\infty}^0 \phi(u)\,du.
\]
Adding the two limits we get the required formula.
If $L_+=L_-$, then the displayed limit becomes
\[
\lim_{\varepsilon\to0^+}\langle T_f',\phi_{x_0,\varepsilon}\rangle
=L_+\int_{\mathbb{R}}\phi(u)\,du,
\]
so $T_f'$ has the full {\L}ojasiewicz point value $L_+$ at $x_0$.
Conversely, assume that $T_f'$ has a full {\L}ojasiewicz point value $L$ at $x_0$.
Choose $\phi_+\in\mathcal{D}_+(\mathbb{R})$ and $\phi_-\in\mathcal{D}_-(\mathbb{R})$ such that
$\int_{\mathbb{R}}\phi_\pm(u)\,du=1$.
Applying the displayed formula first to $\phi_+$ and then to $\phi_-$ gives
$L_+=L$ and $L_-=L$.
Hence $L_+=L_-$.
\end{proof}

\begin{corollary}
If $BD_1f(x_0)$ exists finitely, then $T_f'$ has the {\L}ojasiewicz distributional
point value $BD_1f(x_0)$ at $x_0$.
\end{corollary}

\begin{remark} remark ar prothom dikta
Assume that the one-sided Borel derivatives exist finitely and write
$BD_1^+f(x_0)=L_+$ and $BD_1^-f(x_0)=L_-$.
Then averaging the identities
\[
\frac{f(x_0+t)-f(x_0-t)}{2t}
=\frac12\left(\frac{f(x_0+t)-f(x_0)}{t}
+\frac{f(x_0)-f(x_0-t)}{t}\right)
\]
and
\[
\frac{f(x_0+t)+f(x_0-t)-2f(x_0)}{t}
=\frac{f(x_0+t)-f(x_0)}{t}
-\frac{f(x_0)-f(x_0-t)}{t}
\]
gives
\[
SBD^1f(x_0)=\frac{L_++L_-}{2}.
\]
It also shows that Borel smoothness at $x_0$ is equivalent to $L_+=L_-$.
In particular, if the first-order Borel derivative $BD_1f(x_0)$ exists, then $SBD^1f(x_0)$
also exists and $f$ is Borel smooth at $x_0$.

From the point-value point of view, the first-order Borel derivative is stronger than the first
symmetric Borel derivative.
For $SBD^1$, the conclusion is in general only a symmetric distributional point value, and Borel
smoothness is used here to pass to the full {\L}ojasiewicz point value.
In contrast, for the first-order Borel derivative, the equality of the right and left Borel averages
is already enough to give the full distributional point value.

For example, let $f(x)=x_+=\max\{x,0\}$.
Then
\[
BD_1^+f(0)=1,
\qquad
BD_1^-f(0)=0.
\]
Hence $T_f'=H$ has right distributional point value $1$ and left distributional point value $0$
at $0$, but it has no full {\L}ojasiewicz point value there.
This example shows why the equality $BD_1^+f(x_0)=BD_1^-f(x_0)$ is essential.
\end{remark}

\subsection{The second symmetric Borel derivative}

\begin{theorem}\label{th:SBD2-dpv}
Let $f\in L^1_{\mathrm{loc}}(\mathbb{R})$ and let $x_0\in\mathbb{R}$.
If $SBD^{2}f(x_0)$ exists finitely, say $SBD^{2}f(x_0)=M$, then for every even
$\phi\in\mathcal{D}(\mathbb{R})$,
\[
\lim_{\varepsilon\to 0^+}\langle T_f'',\phi_{x_0,\varepsilon}\rangle
=M\int_{\mathbb{R}}\phi(u)\,du.
\]
In particular, $T_f''$ has symmetric distributional point value $M$ at $x_0$.
\end{theorem}

\begin{proof}
Fix an even $\phi\in\mathcal{D}(\mathbb{R})$ and let $W_2$ be as in
Lemma \ref{lem:kernel-representation}.
Put $g(t)=D_2(t;x_0;f)$ for $t>0$.
By definition of $SBD^{2}f(x_0)$,
\[
\lim_{h\to 0^+}\frac{1}{h}\int_0^h g(t)\,dt=M.
\]
Since $\int_0^\infty W_2(u)\,du=\int_{\mathbb{R}}\phi(u)\,du$,
Lemma \ref{lem:abelian} gives
\[
\lim_{\varepsilon\to 0^+}\int_0^\infty g(\varepsilon u)W_2(u)\,du
=M\int_{\mathbb{R}}\phi(u)\,du.
\]
By Lemma \ref{lem:kernel-representation}, the integral on the left equals
$\langle T_f'',\phi_{x_0,\varepsilon}\rangle$.
This proves the theorem.
\end{proof}

\begin{proposition}\label{prop:second-order-sharp}
The conclusion of Theorem \ref{th:SBD2-dpv} cannot, in general, be upgraded to a full
{\L}ojasiewicz point-value statement.
Indeed, for $f(x)=x_+^2$ one has
\[
SBD^{1}f(0)=0,
\qquad
SBD^{2}f(0)=1,
\]
but $T_f''=2H$ has no {\L}ojasiewicz point value at $0$.
\end{proposition}

\begin{proof}
For $t>0$ we have
\[
D_1(t;0;f)=\frac{t^2-0}{2t}=\frac{t}{2},
\qquad
D_2(t;0;f)=\frac{t^2+0-0}{t^2}=1.
\]
Hence $SBD^{1}f(0)=0$ and $SBD^{2}f(0)=1$.
Since $f'=2x_+$ almost everywhere and $2x_+$ is continuous,
its distributional derivative is the regular distribution $T_f''=2H$.
Therefore, for any $\phi\in\mathcal{D}(\mathbb{R})$,
\[
\langle T_f'',\phi_{0,\varepsilon}\rangle
=2\int_0^\infty \phi(u)\,du.
\]
For even $\phi$ this equals $\int_{\mathbb{R}}\phi(u)\,du$, so the symmetric point value is $1$,
exactly as predicted by Theorem \ref{th:SBD2-dpv}.
However, the limit depends on the odd part of $\phi$, and therefore $T_f''$ has no full
{\L}ojasiewicz point value at $0$.
\end{proof}

\subsection{Point-supported singularities and examples}

\begin{proposition}\label{prop:delta-parity} proposition ar statement r proof ar prothomta
For each $j\in\mathbb{N}\cup\{0\}$ one has
\[
\langle \delta_{x_0}^{(j)},\phi_{x_0,\varepsilon}\rangle
=(-1)^j\varepsilon^{-j-1}\phi^{(j)}(0),
\qquad \phi\in\mathcal{D}(\mathbb{R}).
\]
Consequently:
\begin{enumerate}
\item[(i)] $\delta_{x_0}^{(2m)}$ has no symmetric distributional point value at $x_0$ for any $m\ge 0$;
\item[(ii)] $\delta_{x_0}^{(2m+1)}$ has symmetric distributional point value $0$ at $x_0$ for any $m\ge 0$;
\item[(iii)] $\delta_{x_0}^{(j)}$ has no full {\L}ojasiewicz point value at $x_0$ for any $j\ge 0$.
\end{enumerate}
\end{proposition}

\begin{proof}
By definition of distributional derivatives and of the scaled test function,
\[
\langle \delta_{x_0}^{(j)},\phi_{x_0,\varepsilon}\rangle
=(-1)^j(\phi_{x_0,\varepsilon})^{(j)}(x_0)
=(-1)^j\varepsilon^{-j-1}\phi^{(j)}(0).
\]
If $\phi$ is even and $j$ is odd, then $\phi^{(j)}(0)=0$, which proves (ii).
If $\phi$ is even and $j$ is even, choose an even test function with
$\phi^{(j)}(0)\neq 0$. Then the displayed quantity is a nonzero constant
multiple of $\varepsilon^{-j-1}$, and hence it has no finite limit as
$\varepsilon\to0^+$. This proves (i).

For (iii), choose any test function $\phi$ with
$\int_{\mathbb{R}}\phi(u)\,du=0$ and $\phi^{(j)}(0)\neq 0$.
Then the right-hand side of the {\L}ojasiewicz point-value formula would be
$0$, whereas the left-hand side is again a nonzero constant multiple of
$\varepsilon^{-j-1}$ and therefore has no finite limit. Thus no
$\delta_{x_0}^{(j)}$ has a full point value.
\end{proof}

\begin{remark}
Proposition \ref{prop:delta-parity} shows why parity matters here.
Even test functions detect point-supported singularities of even order, but they
do not detect those of odd order. Together with Proposition
\ref{prop:second-order-sharp}, this explains why Theorem \ref{th:SBD2-dpv} is
only a symmetric point-value statement. Without additional assumptions, no full
second-order point value follows.
\end{remark}
The preceding propositions already give the examples needed to show the limits of the main results.
We close with two elementary examples. The first shows that the symmetric Borel derivatives agree with
the ordinary derivatives in the smooth case. The second shows how a nonzero Dirac mass in $T_f''$
prevents the existence of a finite second symmetric Borel derivative.

\medskip\noindent
\textbf{Example 1.}
If $f\in C^1$ in a neighborhood of $x_0$, then $D_1(t;x_0;f)\to f'(x_0)$ as $t\to 0$,
hence $SBD^{1}f(x_0)=f'(x_0)$.
If $f\in C^2$ in a neighborhood of $x_0$, then $D_2(t;x_0;f)\to f''(x_0)$ as $t\to 0$,
hence $SBD^{2}f(x_0)=f''(x_0)$.

\medskip\noindent
\textbf{Example 2.}
Let $f(x)=|x|$.
Then, at $x_0=0$,
\[
D_1(t;0;f)=0,
\qquad
D_2(t;0;f)=\frac{2}{t}
\qquad (t>0).
\]
Hence $SBD^{1}f(0)=0$, but
\[
\frac{1}{h}\int_0^h D_2(t;0;f)\,dt
=
\frac{2}{h}\int_0^h \frac{dt}{t}
\]
diverges, so $SBD^{2}f(0)$ does not exist.
This agrees with the distributional identity $T_f''=2\delta_0$:
by Proposition \ref{prop:delta-parity}.
\section{Conclusion} 
The paper shows that Borel-type generalized derivatives carry precise distributional information at
a point.
The first and second symmetric Borel derivatives give symmetric distributional point values of
$T_f'$ and $T_f''$, respectively.
In the first-order symmetric case, Borel smoothness is enough to pass from the symmetric point
value to the full {\L}ojasiewicz point value of $T_f'$.

The ordinary Borel derivative gives a complementary picture.
Its one-sided forms determine the right and left distributional point values of $T_f'$.
When the two one-sided Borel averages agree, no additional smoothness condition is needed: the
full distributional point value follows.
This makes clear the difference between the ordinary and symmetric Borel settings.

The examples also mark the limits of the results.
They show that a symmetric point value need not be a full point value, that point-supported
singularities are controlled by parity under even testing, and that the second-order symmetric result
cannot be upgraded automatically.
Thus the main contribution is the bridge between two established theories: Borel-type generalized
derivatives on one side and classical distributional point values on the other.

\end{document}